\documentclass[11pt]{article}

\usepackage{setspace}

\RequirePackage{amsthm,amsmath}
\usepackage{epsfig}
\usepackage{amssymb}
\usepackage{graphics}

\newtheorem{deff}{Definition}[section]

\newtheorem{lem}[deff]{Lemma}
\newtheorem{prop}[deff]{Proposition}

\newtheorem{corr}[deff]{Corollary}

\newcommand{\comment}[1]{}

\newcommand{\PP}[0]{\mathbf{P}}

\begin{document}
 
\title{The exact asymptotic  for the stationary distribution of some unreliable systems.}
\author{Pawe{\l} Lorek  
\\{\it  University of Wroc{\l}aw}
}
\date{September 1, 2009}
 \maketitle


\begin{abstract}
In this paper we find asymptotic distribution for some unreliable networks. Using Markov Additive Structure and Adan, Foley, McDonald \cite{AdaFolMcD08} method, we find the exact asymptotic for the stationary distribution. With the help of MA structure and matrix geometric approach, we also investigate the asymptotic when breakdown intensity   is small. In particular, two different asymptotic regimes for small breakdown intensity suggest two different patterns of large deviations, which is confirmed by simulation study. \par 


\end{abstract}

\section{Introduction}

In this paper we consider problem of finding an exact asymptotic of some non-standard queueing systems.
 We consider two models: Model 1 is an unreliable $M/M/1$  system, and Model 2 is a   system consisting of 2 servers: server 1 being unreliable and server 2 reliable one with a  possible feedback from server 1 to 2. There are a lot of practical problems modelled as Markov chains using 2 or 3 variables. Explicit formulas for stationary distribution can be found only in some special cases. That is why studies on asymptotic for such stationary distributions have been actively conducted by theoretical- and application-oriented researches. There are several techniques available, starting with matrix geometric approach (Neuts, \cite{Neu_book}), matrix analytic method (for recent related work see 
Liu, Miyazawa, Zhao \cite{LMZ07}, Miyazawa, Zhao \cite{MiyZha04}   and 
Tang, Zhao \cite{TanZha04}) or a method of Adan, Foley, McDonald \cite{AdaFolMcD08} which we are going to use widely in this paper.
 \par
  To describe our results, let us start with a brief description of the models. In the unreliable $M/M/1$ system, the customers arrive  according to a Poisson process with intensity $\lambda$ and are served with intensity $\mu$. Moreover, there is an external Markov process which governs the breakdowns and repairs: with intensity $\alpha$ the server can change status from  $Up$ to $Down$, and with intensity $\beta$ vice-versa. While server is in $Down$ status, customers are no longer served, but new ones can join the queue at the server.
\par 
  Model 2 consists of two servers: Customers arrive to the server 2, which is reliable, according to a Poisson process with intensity $\lambda$ and after being served they are directed to the unreliable server 1 which is as the one described above. The service rate at both servers is $\mu$. After being served at server 1 the customer leaves the network with probability $p$ and with probability $1-p$ is rerouted back to the queue at server 1. 

\par 
  The situation described above is different than   {\sl the loss  regime}. In this regime, customers arriving when a server is in $Down$ status, are lost (to the $Down$ server). In \cite{SauDad03}, Sauer and Daduna  showed that under this regime  the stationary distribution of network of unreliable servers is  of product form: the stationary distribution of a pure Jackson network and the stationary distribution of the breakdowns/repairs process. In such a network when customer arrives while server is in $Down$ status, it is lost to the server, but not to the network: it is rerouted - according to some routing regime - to some other server which is in $Up$ status. 
\par

 However, if a customer can join the queue while the server is broken, then the stationary distribution is not of product from, as can be seen in  White and Christie \cite{WhiChr58}. There, the authors give the stationary distribution of Model 1 only. For the Model 2 we are not aware of any results, neither exact distribution, nor asymptotic one.
\par 

In our paper, we give exact asymptotic for both models, following the   method of  Adan, Foley, McDonald \cite{AdaFolMcD08}.
Using Markov Additive Process approach we can clearly show all the differences between two models. 

\par 
 We also consider the behaviour of the ``limiting system'', i.e. the system in which the  breaking probability  $\alpha$ goes to 0. From the method of the above authors we are able to conclude   the exact asymptotic of such ``limiting system'',  but only for   Model 1 and for  some set of parameters: when $\mu<\lambda+\beta$. It turns out to be the same (up to a constant) as the stationary distribution of a $M/M/1$ queue. For the other set of parameters ($\mu>\lambda+\beta$), this method does not lead to a valid asymptotic. However, using the matrix geometric approach (Neuts, \cite{Neu_book}) we show
that then the ``limiting system'' for Model 1 still has the same (up to a constant) stationary distribution as $M/M/1$ queue. The matrix geometric approach, however,  does not give us any information about constants.
Nevertheless we conclude  two different ways in which the system can accumulate a big number of customers. When $\mu<\lambda+\beta$, then in most cases a path leading to a big queue is to be in $Up$ status, and to accumulate a big number of customers, exactly like in standard $M/M/1$ queue (the system does not manage to service customers). The breakdowns/repairs of the system do not have big influence on large deviations. However, if $\mu>\lambda+\beta$, then in most cases a big number of customers is accumulated while the system is in $Down$ status. We illustrate it with simulations (for small $\alpha$), see Figure \ref{fig:1serv_unr} and description on page \pageref{desc:two_ldp} for details.
\par
Furthermore, for Model 2 the method of Adan, Foley, McDonald \cite{AdaFolMcD08} does not lead to a valid asymptotic for the ``limiting system''. Also, matrix geometric approach is not applicable.
\par 
 For the related work, but using different technique see for example Liu, Miyazawa, Zhao \cite{LMZ07}, Miyazawa, Zhao \cite{MiyZha04}   and Tang, Zhao \cite{TanZha04}. Authors therein use matrix analytic method. It uses the fact that some stationary distributions can be presented in matrix form and shown to be solutions of Markov renewal equation, this way decay rates are considered. 
\par In Section 2 we give detailed description of both models, present and discuss all the results. The proofs are in Section 3.

\section{Unreliable server systems  and results}\label{sec:results}
\subsection{Description of systems}\label{sec:description}
 Model 1 is  a following  system consisting of 1 server: customers arrive according to an external Poisson arrival stream with intensity $\lambda$ and are served according to the First Come First Served (FCFS) regime. Each of them requests a service which is exponentially distributed with mean 1. Service is provided with intensity $\mu$.
    There is an external process on the state space $\{Up,Down\}$: with intensity $\alpha$ the server changes status form $Up$ to $Down$ and with intensity $\beta$ from $Down$ to $Up$; $Down$-to-$Up$ and $Up$-to-$Down$ times are exponentially distributed. When the server is broken it immediately stops service, the customer being served is redirected back to the queue. When a new customer arrives while the server is in the $Down$ status, it joins the queue at the server. We assume that all service times, inter-arrivals time, $Down$-to-$Up$ and $Up$-to-$Down$ times constitute an independent family of random variables.
If number of customers is strictly positive, then the transition intensities are as   depicted in Figure \ref{fig:1serv_unr}.
\bigskip\par
  \begin{figure}[h]	
  \begin{center}
\includegraphics{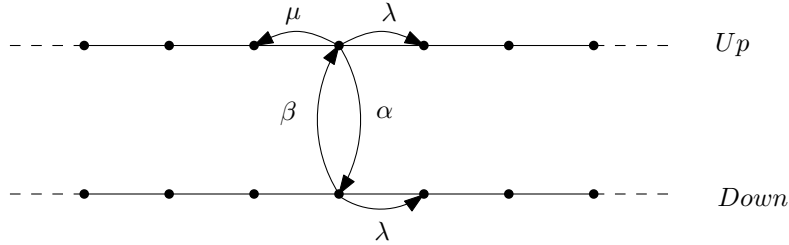}
  \caption{Transitions of Model 1: Unreliable single server.}\label{fig:1serv_unr}
    \end{center}
  \end{figure}
\noindent
Otherwise, if the number of customers is 0, then the transition intensities are similar, except there is no transition from $(0,Up)$ to $(-1,Up)$.
\par 
Model 2 consists of 2 servers. The customers arrive to the reliable  server 2 according to a Poisson process with intensity $\lambda$ and are served there with intensity $\mu$. After being served they are directed to the unreliable server 1, which is exactly unreliable single server system described in Model 1. The service intensity at both systems is $\mu$. After being served at server 2 the customer leaves the network with probability $p$ and with probability $1-p$ it   is rerouted back to join the queue at server 2. The system  is depicted  in Figure \ref{fig:2serv}.

 \begin{figure}[h]	
  \begin{center}
\includegraphics{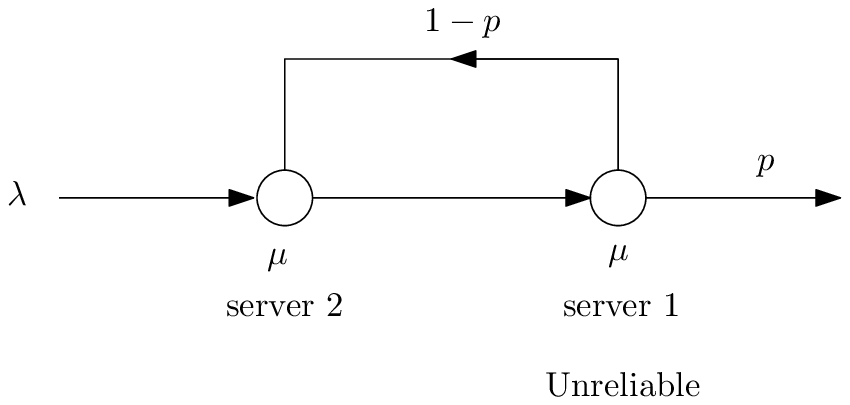}
  \caption{Model 2.}\label{fig:2serv}
    \end{center}
  \end{figure}

\noindent 

    Let $X^{(1)}(t)$ denotes the number of customers present at the server in Model 1 at time $t\geq 0$, either waiting or in service, and let $\sigma(t)\in\{Up,Down\}$ denotes the status of the server. 
 Similarly for Model 2:   $X^{(2)}(t)$ denotes the number of customers present at server 1, $Y^{(2)}(t)$ the number of customers at server 2 and $\sigma(t)$ the status of the unreliable server 1, all at time $t\geq 0$.
We denote the process of Model 1 by $\mathbf{Z}^{(1)}= \{(X(t)^{(1)},\sigma(t)),t\geq 0\}$ and the process of Model 2 by $\mathbf{Z}^{(2)}= \{(X(t)^{(2)},Y^{(2)}(t),\sigma(t)),t\geq 0\}$ . The state space of $\mathbf{Z}^{(1)}$ is $E^{(1)}=\{(x,\sigma),\ x\in\mathbb{N},\ \sigma\in\{Up,Down\}\}$ and the state space of $\mathbf{Z}^{(2)}$ is $E^{(2)}=\{(x,y,\sigma),\ x,y \in\mathbb{N},\ \sigma\in\{Up,Down\}\}$. In the following, the superscripts \ ${}^{(1)}, {}^{(2)}$  denote that constant/number is associated with Model 1 or Model 2 respectively. To have concise notation, we  identify $\{Up,Down\}$ with $\{U ,D \}$. 

\noindent
Throughout the paper we assume that $p>0$ and that the system is not trivial, i.e.
$$\lambda>0, \qquad \mu>0, \qquad \alpha>0, \qquad \beta>0.$$
\noindent 
 Moreover we assume that 
\begin{equation}\label{eq:stability}
\lambda<{\beta\over\alpha+\beta}\mu p,
\end{equation}
which implies that both systems are stable (for Model 1 we mean that condition holds with $p=1$). Actually for Model 1 and for Model 2 with $p=1$ it is ``if and only if'' condition. See  Lemma \ref{lem:stability} for details. 

We can consider ${\beta\over\alpha+\beta}\mu$ as the effective service rate of the unreliable server.
Stability in this case means that we have the unique stationary distribution, which we denote by $\pi$. It will be clear from the context whether $\pi$ is associated with Model 1 or Model 2.
\par 
By $n_k\sim m_k$ we mean that $n_k/m_k \to 1$ as $k\to\infty$. In this paper, ``the exact asymptotic of $\pi$'' means deriving an asymptotic expression for $\pi(k,\sigma)$ (Model 1) or $\pi(k,y,\sigma)$ (Model 2), that is, deriving an expression of the form 
$\pi(k,\sigma)\sim f_k$ or $\pi(k,y,\sigma)\sim g_k$. 

\par 

\noindent It is convenient to define some constants in this place. Let $s_p=(\mu p-\lambda-\beta-\alpha)^2+4\alpha\mu p$. \par 
\noindent Define also
$$\gamma_p={2\lambda\over \lambda+\beta+\mu p+\alpha-\sqrt{s_p}}\in (0,1)$$ 
and
$$
G=\left({\lambda+\beta-\mu-\alpha+\sqrt{s_1}\over 2}+{2\alpha\beta\over\lambda+\beta-\mu-\alpha+\sqrt{s_1}}\right).
$$
\subsection{Results for Model 1}
Our first result is following.
\begin{prop}\label{prop:main}
 
Assume that (\ref{eq:stability}) holds with $p=1$. For the unreliable server system (Model 1)  we have
  
\begin{eqnarray*}
 \pi(k,Up)& \sim & C^{(1)}(Up)   \gamma_1^k , \\[6pt]
 \pi(k,Down)& \sim & C^{(1)}(Down)  \gamma_1^k ,  
\end{eqnarray*}
where
$$C^{(1)}(Up)={\eta^{(1)} \over \tilde{d}^{(1)}} {1\over G} {\lambda+\beta-\mu-\alpha+\sqrt{s_1}\over 2}\neq 0,\quad 
C^{(1)}(Down)={\eta^{(1)}\over \tilde{d}^{(1)}} {\alpha\over G}\neq 0 , $$
$\tilde{d}^{(1)}$ is defined in (\ref{eq:drift_d}) and $\eta^{(1)}$  is equal to $\eta$ in (\ref{eq:cons_eta}) defined for appropriate process.

\end{prop}

\noindent
{\bf Remark: Comparison with standard $M/M/1$}. 
Consider    standard $M/M/1$ queue with arrival and service intensities $\lambda_0, \mu_0$ given by   $\lambda_0=\lambda,\ \mu_0={\beta\over\alpha+\beta}\mu$,  i.e.   both systems have the same   effective rates. We can compare the behaviour of both system for large number of customers $k$. For the $M/M/1$ system, the stationary distribution $\pi_0$ is known exactly:
$$\pi_0(l)=\pi_0(0)\left({\lambda_0\over\mu_0}\right)^l={\mu_0\over \mu_0-\lambda_0}\left({\lambda_0\over\mu_0}\right)^l=
{\beta\mu\over\beta\mu-(\alpha+\beta)\lambda} \left({\alpha+\beta\over \beta}\cdot{\lambda\over\mu}\right)^l.$$
Elementary calculation shows that under  (\ref{eq:stability}) with $p=1$
$$\gamma_1={2\lambda\over \lambda+\beta+\alpha+\mu-\sqrt{s_1}} \geq {\alpha+\beta\over \beta}\cdot{\lambda\over\mu}.$$
It means for big $k$, that $\pi_0$ is stochastically greater than $\pi$. We also note  that for $\pi_0$ only the ratio of $\alpha+\beta$ and $\beta$ matters, but this is not true  for $\pi$.
 

\medskip\par
\noindent

\medskip\par
\noindent
{\bf Remark: Limits as the   breaking  probability  $\alpha $ goes to 0.}\label{remark:constants}
The limit of $\gamma_1$ as $\alpha\to 0$ has twofold nature. It depends on the sign of the difference $\mu-(\lambda+\beta)$:
$$\lim_{\alpha\to 0}\gamma_1={2\lambda\over \beta+\mu+\lambda-\sqrt{(\mu-\lambda-\beta)^2}}=
\left\{
\begin{array}{lcl}
  \displaystyle {\lambda\over\mu} & \mathrm{if} & \mu<\lambda+\beta.\\
 \displaystyle {\lambda\over\lambda+\beta} & \mathrm{if} & \mu>\lambda+\beta,\\[11pt]
\end{array}
\right.
$$
Thus, to calculate  the limits of the constants $C^{(1)}(Up)$ and $C^{(1)}(Down)$ as $\alpha\to 0$, we  consider two cases separately:\par\noindent
\begin{itemize}
 \item $\mu<\lambda+\beta$
\begin{equation}\label{eq:constants_less}
\begin{array}{lll}
\displaystyle \lim_{\alpha\to 0} G= \lambda+\beta-\mu, & \quad & \displaystyle \lim_{\alpha\to 0} \tilde{d}^{(1)}= {\mu-\lambda\over C}, \\[7pt]
\displaystyle \lim_{\alpha\to 0} C^{(1)}(Up) = {\eta^{(1)} C\over \mu-\lambda},  & \quad & \displaystyle \lim_{\alpha\to 0} C^{(1)}(Down)=0 . 
\end{array}
\end{equation}

\item $\mu>\lambda+\beta$
$$
\begin{array}{lll}
\displaystyle \lim_{\alpha\to 0} G= {\beta(\mu-\lambda-\beta)\over\lambda+\beta}, & \quad & \displaystyle \lim_{\alpha\to 0} \tilde{d}^{(1)}= {\lambda+\beta\over C}, \\[7pt]
\displaystyle \lim_{\alpha\to 0} C^{(1)}(Up) =0,  & \quad & \displaystyle \lim_{\alpha\to 0} C^{(1)}(Down)=0. 
\end{array}
$$
\end{itemize}
Of course, from Proposition \ref{prop:main} we always have:
$$\lim_{\alpha\to 0}\lim_{k\to\infty} {\pi(k,Up)\over C^{(1)}(Up) \gamma_1^k}=1.$$
Furthermore, if $\mu<\lambda+\beta$, then via (\ref{eq:constants_less}) we have
\begin{corr}
Assume (\ref{eq:stability}) with $p=1$ and $\mu<\lambda+\beta$. Then for Model 1 we have
$$\lim_{k\to\infty} \lim_{\alpha\to 0} {\pi(k,Up)\over C^{(1)}(Up) \gamma_1^k}=1.$$
\end{corr}\noindent
Note, that in this case $\lim_{\alpha\to 0}\gamma_1={\lambda\over \mu}$, thus $({\lambda\over\mu})^k$ is the asymptotic for the ``limiting system''.  However, if $\mu>\lambda+\beta$, then $\lim_{\alpha\to 0} \gamma_1={\lambda\over\lambda+\beta}$, but $({\lambda\over\lambda+\beta})^k$ is not a correct asymptotic, because  both constants $C^{(1)}(Up)$ and $C^{(1)}(Down)$ have limits 0. In this case the asymptotic for the ``limiting system'' cannot be recovered from Proposition \ref{prop:main}. \par 
\noindent 
However, using matrix geometric approach, we  have the following result.
\begin{prop}\label{prop:main_geometric}
Assume (\ref{eq:stability}) with $p=1$ and $\mu>\lambda+\beta$. Then for Model 1 we have
$$\pi(k,Up)\sim C^{(1)}(Up)\gamma_1^k + C(Up)\gamma^k$$ 
and
$$\lim_{k\to\infty} \lim_{\alpha\to 0}  {\pi(k,Up)\over C(Up)\gamma^k}=1,$$
where $$\gamma={2\lambda\over \lambda+\beta+\mu+\alpha+\sqrt{s_1}},\qquad C(Up)>0.$$
\end{prop}\noindent
Note, that $\gamma$ and $\gamma_1$ differ only by the sign at $\sqrt{s_1}$ and that for $\mu>\lambda+\beta$  we also have $  \lim_{\alpha\to 0}\gamma={\lambda\over \mu}$, so that the asymptotic for small $\alpha$ is still $({\lambda\over\mu})^k$ (we do not have any information about constant a $C(Up)$).
\medskip\par\noindent  
\par\noindent
{\bf Remark: Large deviation path.} Propositions \ref{prop:main} and \ref{prop:main_geometric} suggest two different large deviations paths for small $\alpha$.
The way a large deviation path appears depends on the sign of the difference $\mu-(\lambda+\beta)$:
\begin{itemize}\label{desc:two_ldp}
 \item For $\mu<\lambda+\beta$  the most probable path leading to a big queue is to be more often in the $Up$ status, and to accumulate a lot of customers, because service rate is not big enough. This is exactly the way it appears in standard M/M/1 queue.
\item For $\mu>\lambda+\beta$ the service rate $\mu$ is big enough, so that large deviation path does not appear in the standard way: in this case the most probable situation is, that a lot of customers join the queue, when the server is almost entirely in the  $Down$ status. 
\end{itemize}
We illustrate this behaviour in Figure \ref{fig:large_dev_path} below: The  plots are for both cases: $\mu<\lambda+\beta$ and $\mu>\lambda+\beta$; $x-$axis is the step number, $y-$axis is the number of customers, 'dot' denotes that the server was in $Up$ status and 'cross' denotes that the server was in $Down$ status. For each case there are two plots: one with steps ranging from 0 to 70000 and second with steps chosen in such the way, so that a large deviation path is well depicted. In case $\mu<\lambda+\beta$ there is also depicted a line with slope of the large deviation path given by $\tilde{d}^{(1)}$  in (\ref{eq:drift_d}).

\begin{center}
  \begin{figure}[h]	
 \begin{tabular}{c}
\includegraphics[width=370pt]{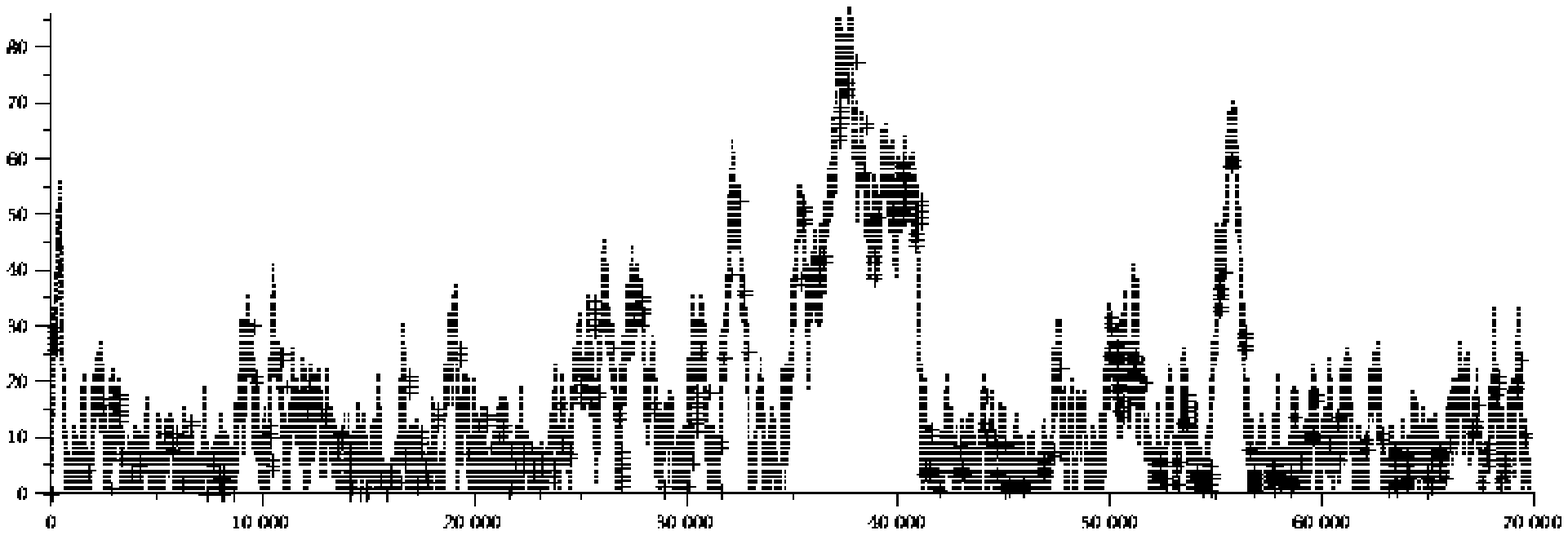}\\
\footnotesize $\mathrm{Steps:\ from\ 0 \ to \ 70000} $\\
\includegraphics[width=370pt,height=120pt]{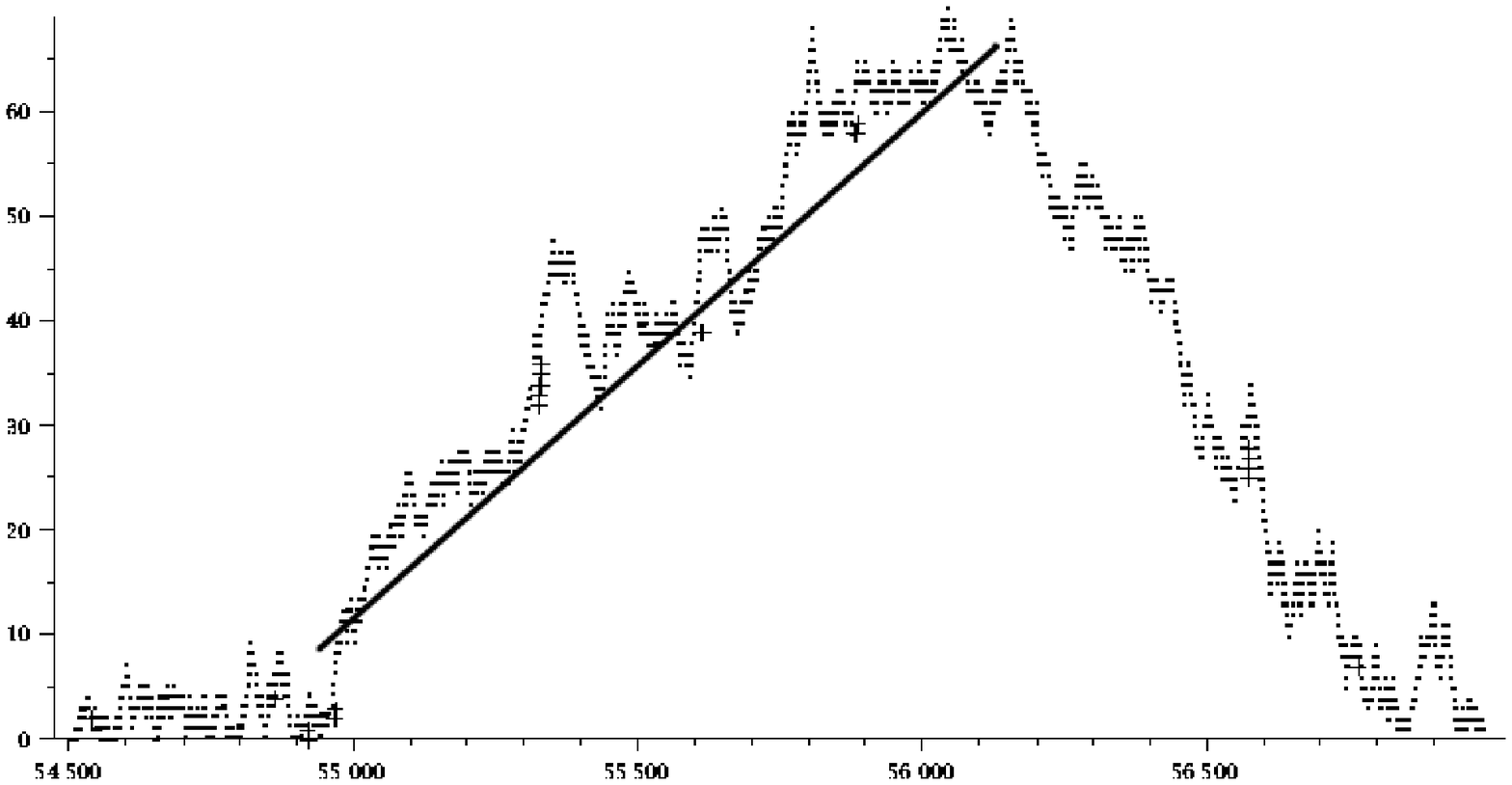}\\
\footnotesize $ \mathrm{Steps:\ from\ 54550\ to\ 57000} $\\
\small a) $\mu<\lambda+\beta:\qquad \alpha=0.1, \beta=10, \lambda=10, \mu=11. $\\
\vspace{10pt} \hspace{10pt}  \\
\includegraphics[width=370pt]{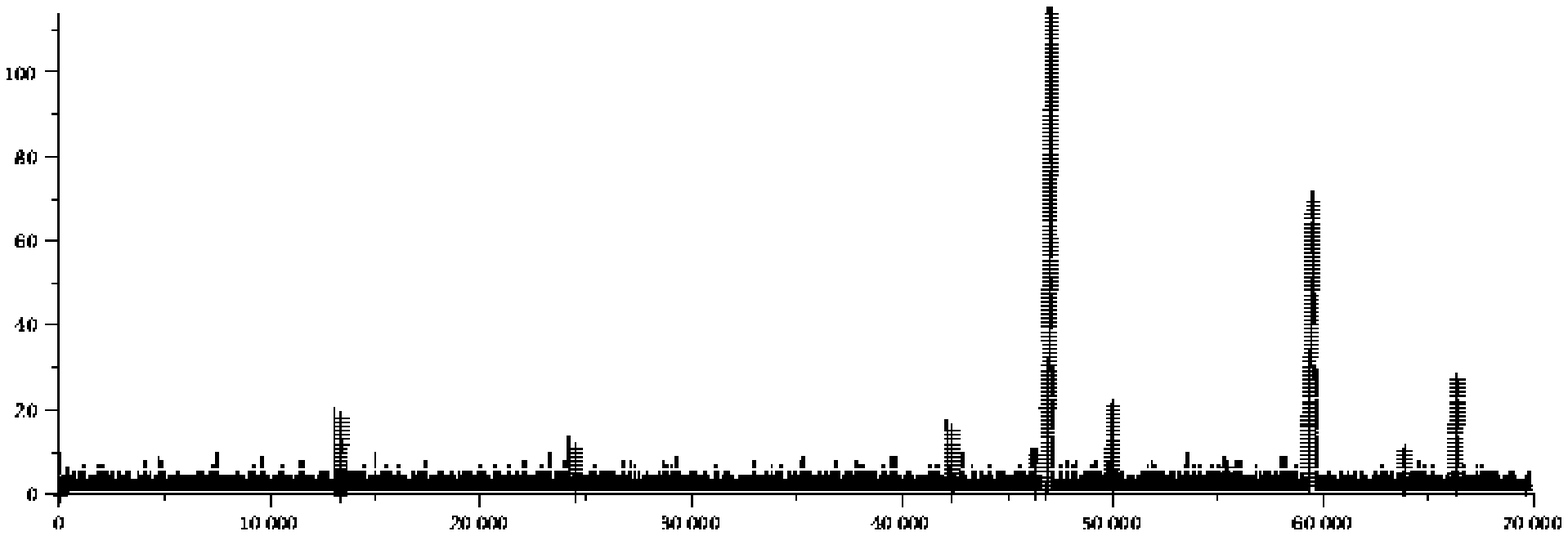} \\
\footnotesize $\mathrm{Steps:\ from\ 0 \ to \ 70000} $\\
\includegraphics[width=370pt,height=120pt]{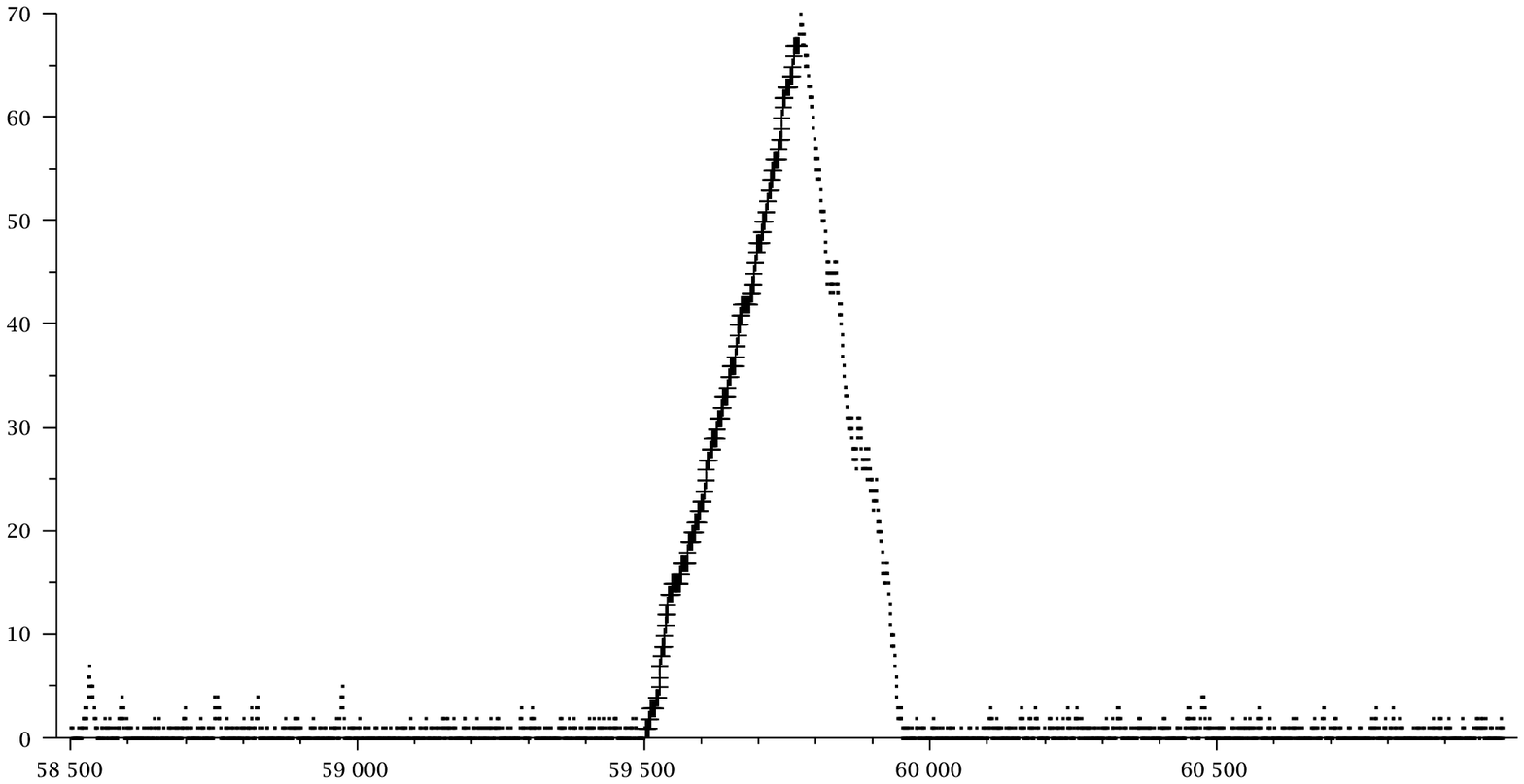} \\
\footnotesize $\mathrm{Steps:\ from\ 58500\ to\ 61000} $\\
\small b) $\mu>\lambda+\beta:\qquad \alpha=0.01, \beta=1, \lambda=20, \mu=60. $\\
\end{tabular}
\caption{Two different  large deviation paths}\label{fig:large_dev_path}
\end{figure}
 \end{center}

\subsection{Results for Model 2}
For the general Model 2 with $p\in(0,1)$ we have the following result about exact asymptotic, although we do not have knowledge about the constants.  
\par\noindent
\begin{prop}\label{prop:main2p}
 
Assume that (\ref{eq:stability}) holds. For Model 2  we have
  
\begin{eqnarray*}
 \pi(k,y,Up)& \sim & C(Up,y) \gamma_p^k, \\[6pt]
 \pi(k,y,Down)& \sim & C(Down,y)  \gamma_p^k,
\end{eqnarray*}
where $C(Up,y)>0,\ C(Down,y)>0$.
\end{prop}
 
\noindent 
{\bf Remark: Limits as the  breaking probability  $\alpha$ goes to 0}. The limit of $\gamma_p$ as $\alpha\to 0$ has again twofold nature, it depends on the sign of the difference $\mu p-(\lambda+\beta)$.
$$\lim_{\alpha\to 0}\gamma_p={2\lambda\over \beta+\mu p+\lambda-\sqrt{(\mu p-\lambda-\beta)^2}}=
\left\{
\begin{array}{lcl}
  \displaystyle {\lambda\over\mu p} & \mathrm{if} & \mu p<\lambda+\beta.\\
 \displaystyle {\lambda\over\lambda+\beta} & \mathrm{if} & \mu p>\lambda+\beta,\\[11pt]
\end{array}
\right.
$$
Unfortunately, we do not have any information about constants $C(Up,y), C(Down,y)$. In particular, we do not know if the limits of them are positive (in Model 1 in one case the constant $C^{(1)}(Up)$ was positive, while in the other it was 0). It  means that from Proposition \ref{prop:main2p} we cannot recover the asymptotic for the ``limiting system''.  

\par\medskip 
\par\noindent 
\par  
For Model 2 with $p=1$ (which is the tandem of reliable and unreliable servers) we have the following exact asymptotic result.
\par 
\begin{prop}\label{prop:main2}
 
Assume that (\ref{eq:stability}) holds with $p=1$. For the tandem system with unreliable server 1 (i.e.  Model 2 with $p=1$) we have
  
\begin{eqnarray*}
 \pi(k,y,Up)& \sim & C^{(2)}(Up) \left({\lambda\over \mu}\right)^y \gamma_1^k, \\[6pt]
 \pi(k,y,Down)& \sim & C^{(2)}(Down) \left({\lambda\over \mu}\right)^y \gamma_1^k,
\end{eqnarray*}
where 
$$C^{(2)}(Up)={\eta^{(2)} \over \tilde{d}^{(2)}} {1\over G} {\lambda+\beta-\mu-\alpha+\sqrt{s_1}\over 2} \cdot B\neq 0,\quad C^{(2)}(Down)={\eta^{(2)} \over \tilde{d}^{(2)}}{ \alpha \over G} \cdot B \neq 0 ,\quad  $$

$$B=1-{\lambda+\beta+\mu+\alpha-\sqrt{s_1}\over 2\mu},$$
 $\tilde{d}^{(2)}$ is defined in (\ref{eq:drift_d2}) and  $\eta^{(2)}$ is equal to $\eta$ in (\ref{eq:cons_eta}) defined for appropriate process. 
\end{prop}
\noindent 
{\bf Remark: Limits as the  breaking probability  $\alpha$ goes to 0}. 
Note that $C^{(1)}(Up)$, $C^{(2)}(Up)$,  and $C^{(1)}(Down)$, $C^{(2)}(Down)$ differ only by a factor of $B$ and ${\eta^{(1)}\over\tilde{d}^{(1)}}$ or  ${\eta^{(2)}\over\tilde{d}^{(2)}}$. We can rewrite
$C^{(2)}(Up)={\eta^{(2)}\over\tilde{d}^{(2)}} {\tilde{d}^{(1)} \over \eta^{(1)}} \cdot B \cdot C^{(1)}(Up)$ and similarly with $C^{(2)}(Down)$. Moreover, $\tilde{d}^{(1)}$ and $\tilde{d}^{(2)}$ are different, but they have the same limits as $\alpha\to 0$. Thus, based on results for  Model 1 and calculating the limit of $B$, we have
two cases:\par\noindent
\begin{itemize}
 \item $\mu<\lambda+\beta$
\begin{equation*}
\begin{array}{lll}
\displaystyle \lim_{\alpha\to 0} B= 0, & \quad & \displaystyle \lim_{\alpha\to 0} \tilde{d}^{(2)}= {\mu-\lambda\over C}, \\[7pt]
\displaystyle \lim_{\alpha\to 0} C^{(1)}(Up) = {\eta^{(1)} C\over \mu-\lambda},  & \quad & \displaystyle \lim_{\alpha\to 0} C^{(1)}(Down)=0 . \\[7pt]
\displaystyle \lim_{\alpha\to 0} C^{(2)}(Up) = 0,  & \quad & \displaystyle \lim_{\alpha\to 0} C^{(2)}(Down)=0 . \\[7pt]
\end{array}
\end{equation*}

\item $\mu>\lambda+\beta$
$$
\begin{array}{lll}
\displaystyle \lim_{\alpha\to 0} B={\mu-(\alpha+\beta)\over \mu}, & \quad & \displaystyle \lim_{\alpha\to 0} \tilde{d}^{(2)}= {\lambda+\beta\over C}, \\[7pt]
\displaystyle \lim_{\alpha\to 0} C^{(1)}(Up) =0,  & \quad & \displaystyle \lim_{\alpha\to 0} C^{(1)}(Down)=0. \\[7pt]
\displaystyle \lim_{\alpha\to 0} C^{(2)}(Up) =0,  & \quad & \displaystyle \lim_{\alpha\to 0} C^{(2)}(Down)=0. 
\end{array}
$$
\end{itemize}
It means that from Proposition \ref{prop:main2} we connote recover the asymptotic for the ``limiting system''. 
For $\mu>\lambda+\beta$ it is because limits of both constants $C^{(1)}(Up)$ and $C^{(1)}(Down)$ (and therefore $C^{(2)}(Up)$ and $C^{(2)}(Down)$)  are 0. In case $\mu<\lambda+\beta$ although the limit of $C^{(1)}(Up)$ is strictly positive, the limit of $C^{(2)}(Up)$ is again 0, because of the limit of $B$.

\section{Proofs}\label{sec:proofs}
\subsection{Uniformization and stability}
We find it more convenient to work with the embedded discrete-time Markov chain. We denote its kernel by $\PP$. Of course it has the same stationary distribution $\pi$. We make uniformization by fixing some $C$ such that $C\geq \lambda+\mu+\alpha+\beta$.

\begin{lem}\label{lem:stability} Model 1 and Model 2 with $p=1$ are ergodic if and only if
$\lambda < {\beta\over \alpha+\beta} \mu. $
Moreover, condition $\lambda < {\beta\over \alpha+\beta} \mu p $ is sufficient for stability of Model 2 with $p\in(0,1)$.
\end{lem}

\begin{proof}
\par \noindent
\begin{itemize}
 \item  Model 1:
\par\noindent
If we order states in the following way
$$(0,U)\prec (0,D)\prec (1,U)\prec (1,D) \prec (2,U)\prec (2,D)\prec \ldots$$
we can rewrite
\begin{equation}\label{eq:matrixP}
\PP=\left(
\begin{array}{ccccc}
\mathbf{P}_1^{(0)} & \mathbf{P}_0 & & & \\
\mathbf{P}_2 & \mathbf{P}_1 & \mathbf{P}_0 & & \\
&\mathbf{P}_2 & \mathbf{P}_1 & \mathbf{P}_0 &  \\
& & \ddots & \ddots &\ddots\\
\end{array}
\right),
\end{equation}
where
\begin{equation}\label{eq:neuts_matrices}
\PP_1^{(0)}=\left(
\begin{array}{cc}
1-{\alpha+\lambda \over C} & {\alpha\over C} \\[5pt]
{\beta\over C} &  1-{\lambda+\beta\over C} \\
\end{array}\right)
,\ 
\PP_0=\left(
\begin{array}{cc}
{\lambda\over C} & 0 \\[5pt]
0 &  {\lambda\over C} \\
\end{array}\right)
,$$
$$
\PP_2=\left(
\begin{array}{cc}
{\mu \over C} & 0 \\[5pt]
0 &  0 \\
\end{array}\right)
,\
\PP_1=\left(
\begin{array}{cc}
1-{\mu+\lambda+\alpha\over C} & {\alpha\over C}\\[5pt]
{\beta\over C} & 1-{\lambda+\beta\over C} \\
\end{array}\right).
 \end{equation}
Therefore $\PP$ is {\sl quasi-birth-and-death process} (QBD process) with inter-level generator
$$\mathbf{G}=\PP_0+\PP_1+\PP_2=\left(
\begin{array}{cc}
1-{\alpha\over C} & {\alpha\over C} \\
 {\beta\over C} &  1-{\beta\over C}\\
\end{array}\right).
$$
From Neuts \cite{Neu_book} (Theorem 3.1.1), we have that if inter-level generator matrix $\mathbf{G}$ is irreducible, then
the process is positive recurrent if and only if
$$\rho \cdot \PP_0 \cdot
\left(\begin{array}{c} 1 \\ 1\end{array}\right)
<
\rho \cdot \PP_2 \cdot
\left(\begin{array}{c} 1 \\ 1\end{array}\right),
$$
where $\rho$ is the stationary probability vector of $\mathbf{G}$.
\par
We have $\rho=\left( {\beta\over\alpha+\beta},{\alpha\over\alpha+\beta}\right)$,
$\rho \cdot \PP_0 \cdot
\left(\begin{array}{c} 1 \\ 1\end{array}\right) =\lambda\cdot {1\over C} $
and
$
\rho \cdot \PP_2 \cdot
\left(\begin{array}{c} 1 \\ 1\end{array}\right) ={\beta\over \alpha+\beta}\mu\cdot {1\over C}$ which finishes the proof.

\item Model 2 with $p=1$:
\par\noindent
The server 2 is  stable if and only if $\lambda<\mu$. The output of server 2 is the Poisson process with intensity $\lambda$ (Burke's Theorem, see Burke \cite{Bur56} for details). In previous case we proved that unreliable server 1 with arrival rate $\lambda$ and service rate $\mu$ is stable if and only if $\lambda < {\beta\over \alpha+\beta} \mu$. Of course the second condition implies first.

\item Model 2 with $p\in(0,1)$: 
\par\noindent
Later, in Section  \ref{sec:Model2_harmonic_fun}, the harmonic function of the (so-called) free process is derived. By Proposition \ref{prop:mcdonald} it gives the following  asymptotic (actually this was given in Proposition \ref{prop:main2p}), for any $k\in \mathbb{N}$ and $\sigma\in\{U,D\}$
$$\pi(k,y,\sigma)\sim C(\sigma,y)\gamma_p^k.$$
It is enough to show, that $\lambda<{\beta\over \alpha+\beta}\mu p $ implies $\sum_{k}\pi(k,y,\sigma)<1$, or equivalently, that $\gamma_p<1$. It can be easily checked, that $\gamma_p<1$ if and only if $\lambda<{\beta\over\alpha+\beta}\mu p$, thus this is a sufficient condition.
\end{itemize}
\end{proof}

\noindent {\bf Remark}. Consider system similar to Model 2, but with 2 reliable servers (i.e. standard Jackson network) with service rate at server 2: $\mu_2=\mu$ and service rate at server 1: $\mu_1={\beta\over\alpha+\beta}\mu  $ (which is the effective service rate of the unreliable server). Then, solving traffic equation and using standard stability conditions for Jackson networks, we have that the system is stable if and only if ${\lambda\over \mu_1 p}<1$, i.e. $\lambda<{\beta\over\alpha+\beta}\mu p$. It suggests that (\ref{eq:stability}) is the necessary stability condition for Model 2 with $p\in(0,1)$.

\subsection{Markov Additive Structure and  result of Adan, Foley and McDonald \cite{AdaFolMcD08}}
Tools used in this paper fall into the framework of  Adan, Foley and McDonald  \cite{AdaFolMcD08}, where Markov additive structure is needed. 
Let $Z_n=(X_n, Y_n)$ be a Markov process with state space $\mathbb{Z}^k\times \mathcal{E}$, where $\mathbb{Z}=\{\ldots,-2,-1,0,1,2,\ldots\}$.
If the transitions are invariant with respect to the translations on $x\in \mathbb{Z}^k$, i.e.:
$$\PP((x,y),(x',y'))=\PP((0,y),(x'-x,y'))\qquad \mathrm{for \ all \ } x, x'\in\mathbb{Z}^k \textrm{ \ and \ } y,y'\in \mathcal{E},$$
then it is called a Markov additive process, $X_n$ is its additive part, $Y_n$ is a Markovian part. \par 
Processes     $\mathbf{Z}^{(1)}$ and $\mathbf{Z}^{(2)}$ defined in Section \ref{sec:description} are  Markov additive   if we remove the boundaries and let the transitions to be shift invariant relative to the first coordinate. Abusing notation, we denote state spaces of these processes with the same symbols, respectively,  $E^{(1)}=\mathbb{Z}\times\{U,D\}$ and $E^{(2)}=\mathbb{Z}\times\mathbb{N}\times\{U,D\}$.
\smallskip\par\noindent

By harmonic function of Markov chain with transition matrix $\PP$ we mean the  right eigenvector $h$ associated with eigenvalue 1, i.e. such that $\PP h=h$.
 From  \cite{AdaFolMcD08} we can deduce the following. 

\begin{prop}\label{prop:mcdonald}
Consider Markov process $\{X_t\}_{t\geq 0}$ with stationary distribution $\pi$ and state space $E=\{(k,A): k\in \mathbb{Z}, A\in\mathbb{Z}^n\}$. Let $\triangle\subset E$ and let $\mathbf{K}^\infty$ be the kernel of the free process, which is shift invariant relative to first coordinate. Let $$\boldsymbol{\mathcal{K}}((k,A),(k',A'))=\mathbf{K}^\infty((k,A),(k',A')) h(k',A')/h(k,A)$$ be the kernel of so-called twisted free process, where $h$ is the harmonic function of $\mathbf{K}^\infty$. If 
\begin{equation}\label{condition:pih}
\sum_{(k,A)\in\triangle} \pi(k,A) h(k,A)<\infty, 
\end{equation}
then
$$\pi(l,A)\sim {\eta \varphi(A)\over \tilde{d}\ h(l,A),}$$
where $\tilde{d}$ is the stationary horizontal drift   and 
\begin{equation}\label{eq:cons_eta}
\eta\equiv \sum_{(x',A')\in\triangle} \pi(x',A') h(x',A') \mathcal{H}(x',A'). 
\end{equation}
$\mathcal{H}(x',A')$ is the probability that twisted free process starting from $(x',A')$ never hits $(E\setminus\triangle)^C$.

\end{prop}

\subsection{Proof of Proposition \ref{prop:main}}
\subsubsection{The free process.}\noindent
 We have to define $\triangle\subset E^{(1)}$ and a Markov additive process embedded in original one, so that it is shift invariant outside the boundary $\triangle$.
We want the process to be additive in the first coordinate and we want the second coordinate to be the Markovian part.
Thus, as a boundary we can take $\triangle=\{(0,Up)\cup (0,Down)\}$. Let us denote the transition kernel of this process by $\mathbf{K}^\infty$.
Being Markov additive in the first coordinate means $\mathbf{K}^\infty((m,\sigma),(z+m,\sigma'))=\mathbf{K}^\infty((0,\sigma),(z,\sigma'))$, where
 $$
\mathbf{K}^\infty((0,\sigma),(z,\sigma'))
=\left\{
\begin{array}{lll}
 {\lambda\over C} & \mathrm{for} \ z=1  \ \mathrm{and} \ \sigma'=\sigma\\[4pt]
{ \mu\over C} & \mathrm{for} \ z=-1 \ \mathrm{and} \ \sigma'=\sigma=U & \\[4pt]
  {\alpha \over C}&\mathrm{for} \ z=0,\ \sigma=U \ \mathrm{and} \ \sigma'=D & \\[4pt]
 {\beta\over C} &\mathrm{for} \ z=0,\ \sigma=D \ \mathrm{and} \ \sigma'=U & \\[4pt]
 1-{\lambda+\beta\over C} & \mathrm{for} \ z=0, \ \mathrm{and} \ \sigma'=\sigma=D \\[4pt]
 1-{\alpha+\mu+\lambda\over C} &\mathrm{for} \ z=0, \ \mathrm{and} \ \sigma'=\sigma=U \\[4pt]
\end{array}
\right. 
$$

\noindent
Since we have removed the boundary, the free process walks over all $\mathbb{Z}\times \{Up, Down\}$.

\subsubsection{Feynman-Kac kernel}\label{sec:feynman-kac}
With the free process we associate the following Feynman-Kac kernel: 
\par\noindent
$\displaystyle \mathbf{K}_\theta(\sigma,\sigma')=\sum_{z} \mathbf{K}^\infty((0,\sigma),(z,\sigma'))e^{\theta z}$, where $\sigma,\sigma'\in\{U,D\}$. 
We have
$$
\mathbf{K}_\theta=\left(
\begin{array}{cc}
 {\lambda\over C} e^{\theta}+1-{\alpha+\mu+\lambda\over C}+{\mu\over C}e^{-\theta}
& {\alpha\over C} \\[5pt]
{\beta\over C} & {\lambda\over C} e^{\theta}+1-{\lambda+\beta\over C}\\
\end{array}
\right).
$$
$\mathbf{K}_\theta$ has two eigenvalues 
$$k_{1,2}(\theta):={1\over C} \left(C-{\alpha\over 2}-{\beta\over 2}-{\mu\over 2}-\lambda+{\mu e^{-\theta}\over 2}+\lambda e^{\theta} \pm {1\over 2} \sqrt{ (\mu e^{-\theta}-\alpha-\beta-\mu)^2-4\mu\beta(1-e^{-\theta})}\right),
$$
We are interested in the larger eigenvalue, i.e. we  only consider $k_1$.
We want the largest eigenvalue to be equal to 1, i.e. $k_1(\theta)=1.$ Set: $t=e^{\theta}$. It means
$$C-{1\over2}(\alpha+\beta+\mu)-\lambda+{\mu\over 2t}+\lambda t+{1\over 2}\sqrt{({\mu\over t} -\alpha-\beta-\mu)^2-4\mu\beta(1-{1\over t})}=C.$$
Equivalently,
\begin{equation}\label{eq:sq_root}
\sqrt{\left({\mu\over t} -\alpha-\beta-\mu\right)^2-4\mu\beta(1-{1\over t})}=\alpha+\beta+\mu+2\lambda-{\mu\over t}-2\lambda t
\end{equation}
To find the solution of the above equation, we have to solve 
\begin{equation}\label{eq:Wt}
W(t):=\lambda^2 t^3-\lambda(\beta+\alpha+2\lambda+\mu) t^2+(\lambda(\alpha+2\mu+\beta+\lambda)+\mu\beta)t-\mu(\alpha+\beta)=0.
\end{equation}
Of course $W(1)=0$, thus
$$W(t)= (t-1)(\lambda^2 t^2-\lambda(\beta+\alpha+\lambda+\mu)t+\mu(\lambda+\beta)).$$
We   obtain two solutions:
\begin{equation}\label{eq:t12}
t_{1}=  {\lambda+\beta+\mu+\alpha + \sqrt{s_1}\over 2\lambda}, \qquad t_2=  {\lambda+\beta+\mu+\alpha-\sqrt{s_1}\over 2\lambda}.
\end{equation}
\noindent
Note, that $t_2=\gamma_1^{-1}$.
We want the right hand side of (\ref{eq:sq_root}) to be positive, what is equivalent to
\begin{equation*}
2\lambda t^2-(\alpha +\beta+\mu+2\lambda)t+\mu  <0.
\end{equation*}
However, one can check (noting, that   $s_1=(\mu +\lambda+\beta+\alpha)^2-4\mu (\lambda+\beta)$) that  $t_1$ is not the solution of (\ref{eq:sq_root}), because then the right hand side of the equation is negative.

\subsubsection{The harmonic function of the free process}
\begin{lem}\label{lem:free_proc_harmonic}
The harmonic function of the free process is the following:
$$  h(x,U)=\left({1\over \gamma_1}\right)^{x},\qquad h(x,D)=\left({1\over \gamma_1}\right)^{x} \cdot {2\beta\over\lambda+\beta-\mu-\alpha+\sqrt{s_1} }.  $$
\end{lem}
\begin{proof}
We want to find the harmonic function for free process of the  form $h(z,\sigma)=t_2^z e^{\theta_\sigma}$, where $t_2$ is such that the largest eigenvalue of Feynman-Kac kernel is equal to one, i.e. 

$$h(z,\sigma)= \left({1\over \gamma_1}\right)^{z} e^{\theta_\sigma}.$$
For $h$ to be the harmonic function for free process we have to have
\begin{equation}\label{eq:harm_cond}
\forall (z\in\mathbb{Z}) \qquad  \sum_{x ,\sigma} \mathbf{K}^\infty ((z,U),(x,\sigma))h(x,\sigma)=h(z,U) \textrm{ and }
\sum_{x,\sigma} \mathbf{K}^\infty ((z,D),(x,\sigma))h(x,\sigma)=h(z,D).
\end{equation}
First part of (\ref{eq:harm_cond}) means
$$
\sum_{x ,\sigma} \mathbf{K}^\infty((z,U),(x ,\sigma))h(x ,\sigma)=
{\lambda\over C}\left({1\over \gamma_1}\right)^{ (z+1)}e^{\theta_U}+{\mu\over C} \left({1\over \gamma_1}\right)^{ (z-1)}e^{\theta_U}+{\alpha\over C} t_2^{z}e^{\theta_D}$$
$$+\left(1-{\lambda+\mu+\alpha\over C}\right)\left({1\over \gamma_1}\right)^z e^{\theta_U}
=h(z,U)=\left({1\over \gamma_1}\right)^{z}e^{\theta_U}
$$
and equivalently
$$e^{\theta_U}\left[1-{\lambda\over C} {1\over \gamma_1} -{\mu\over C} \gamma_1-\left(1-{\lambda+\mu+\alpha\over C}\right)\right]= {\alpha\over C}  e^{\theta_D},$$
i.e.
$$e^{\theta_U}[\lambda+\mu+\alpha -\lambda {1\over \gamma_1} -\mu \gamma_1]= \alpha e^{\theta_D}.$$

\noindent
Second part of (\ref{eq:harm_cond}) means
$$\sum_{x ,I} \mathbf{K}^\infty((z,D),(x ,I))h(x ,I)={\lambda\over C} \left({1\over \gamma_1}\right)^{(z+1)}e^{\theta_D}+{\beta\over C} \left({1\over \gamma_1}\right)^{ z}e^{\theta_U}+\left(1-{\lambda+\beta\over C}\right)\left({1\over \gamma_1}\right)^{ z}e^{\theta_D}$$
$$=h(z,D)=e^{\theta_D}\left({1\over \gamma_1}\right)^{ z} $$
and equivalently
$$e^{\theta_D}\left[1- {\lambda\over C} {1\over \gamma_1}-\left(1-{\lambda+\beta\over C}\right)\right]= {\beta\over C} e^{\theta_U},$$
i.e.
$$e^{\theta_D}\left[\lambda+\beta- \lambda  \left({1\over \gamma_1}\right)\right]= \beta e^{\theta_U}.$$
\noindent
Putting these conditions together we have:
$$\left\{
\begin{array}{lllll}
e^{\theta_U}[\lambda+\mu+\alpha -\lambda {1\over \gamma_1} -\mu \gamma_1]&=& \alpha  e^{\theta_D},&\qquad & (i)\\[5pt]
e^{\theta_D}[\lambda+\beta-  \lambda {1\over \gamma_1}]&=&\beta e^{\theta_U}.&\qquad & (ii)\\
\end{array}\right.
$$
One of $e^{\theta_U}$ or $e^{\theta_D}$ can be arbitrary, set $e^{\theta_U}=1$.
From $(ii)$ we have
$$e^{\theta_D}={\beta\over\lambda+\beta-\lambda {1\over \gamma_1}}
={2\beta\over\lambda+\beta-\mu-\alpha+\sqrt{ s_1} }.$$
\noindent

\end{proof}

\subsubsection{The twisted free process.}
With the harmonic function of the free process we can define the $h$-transform (or twisted kernel) in the following way: $\boldsymbol{\mathcal{K}}((m,\sigma),(z+m,\sigma'))=\boldsymbol{\mathcal{K}}((0,\sigma),(z,\sigma'))=\boldsymbol{\mathbf{K}}^\infty((0,\sigma),(z,\sigma')) {h(z,\sigma')\over h(0,\sigma)}$, i.e.
 $$
\boldsymbol{\mathcal{K}}((0,\sigma),(z,\sigma'))
=\left\{
\begin{array}{llllll}
   {\lambda\over C} {h(1,U)\over h(0,U)} & =&   {1\over C} {\lambda+\beta+\mu+\alpha-\sqrt{s_1}\over 2} &   \mathrm{for} \ z=1,\\[8pt]

   {\mu\over C} {h(-1,U)\over h(0,U)} & =&   {1\over C} {2\lambda\mu \over \lambda+\beta+\mu+\alpha-\sqrt{s_1} } & \mathrm{for} \ z=-1 \ \mathrm{and} \ \sigma=\sigma'=U & \\[8pt]

 {\alpha\over C} {h(0,D)\over h(0,U)} & =&     
 {1\over C} { 2\alpha\beta  \over \lambda+\beta-\mu-\alpha+\sqrt{s_1}}&\mathrm{for} \ z=0,\ \sigma=U \ \mathrm{and} \ \sigma'=D, & \\[8pt]

{\beta\over C} {h(0,U)\over h(0,D)} & =&   
{1\over C} {\lambda+\beta-\mu-\alpha+\sqrt{s_1}\over 2}&\mathrm{for} \ z=0,\ \sigma=D \ \mathrm{and} \ \sigma'=U,& \\[8pt]

 (1-{\lambda+\beta\over C}) {h(0,D)\over h(0,D)} & =&     1-{\lambda+\beta\over C} & \mathrm{for} \ z=0, \ \mathrm{and} \ \sigma=\sigma'=D, \\[8pt]

\left(1-{\lambda+\alpha+\mu\over C}\right) {h(0,U)\over h(0,U)} & =&    1-{\lambda+\alpha+\mu\over C} &\mathrm{for} \ z=0, \ \mathrm{and} \ \sigma=\sigma'=U. \\[8pt]
\end{array}
\right. 
$$

\medskip\par\noindent
The transition diagram is simply a reweighting of the transitions in Figure \ref{fig:1serv_unr}. \par 

\medskip\par\noindent

Now we are interested in the stationary distribution of the Markovian part of the twisted free process, call it $\boldsymbol{\mathcal{K}}_2$, which  state space is $\{U,D\}$. We have:
$$
\begin{array}{lllllll}
\boldsymbol{\mathcal{K}}_2(U,D)&=&\boldsymbol{\mathcal{K}}((0,U),(0,D), & \\
\boldsymbol{\mathcal{K}}_2(D,U)&=&\boldsymbol{\mathcal{K}}((0,D),(0,U), & \\
\boldsymbol{\mathcal{K}}_2(U,U)&=&\boldsymbol{\mathcal{K}}((0,U),(0,U), &+\boldsymbol{\mathcal{K}}((0,U),( -1,U))+&\boldsymbol{\mathcal{K}}(( ,U),( 1,U)
&=1-\boldsymbol{\mathcal{K}}_2(U,D),\\
\boldsymbol{\mathcal{K}}_2(D,D)&=&\boldsymbol{\mathcal{K}}((0,D),(0,D) &+\boldsymbol{\mathcal{K}}((0,D),(1,D))
& &=1-\boldsymbol{\mathcal{K}}_2(D,U).\\
\end{array}
$$
For 2-states Markov chain with transition matrix $ \left(
\begin{array}{cc}
1-p_1 & p_1  \\
p_2 & 1-p_2\\
\end{array}\right)
$
the stationary distribution is $\displaystyle \pi(1)=p_2/(p_1+p_2), \quad \pi(2)=1-\pi(1)=p_1/(p_1+p_2).$\smallskip 
\par
\noindent Let $\varphi$ be the stationary distribution of $\boldsymbol{\mathcal{K}}_2$.
We have
$$\varphi(U)= {\boldsymbol{\mathcal{K}}_2(D,U)\over \boldsymbol{\mathcal{K}}_2(D,U) + \boldsymbol{\mathcal{K}}_2(U,D)}, \qquad \varphi(D)={\boldsymbol{\mathcal{K}}_2(U,D)\over \boldsymbol{\mathcal{K}}_2(D,U) + \boldsymbol{\mathcal{K}}_2(U,D)}.$$
Note that  $G=  C(\boldsymbol{\mathcal{K}}_2(D,U) + \boldsymbol{\mathcal{K}}_2(U,D)) $  and   rewrite
$$\varphi(U)={1\over G} \boldsymbol{\mathcal{K}}_2(D,U)= {1\over G}\cdot   {\lambda+\beta-\mu-\alpha+\sqrt{s_1}\over 2}, $$
$$\varphi(D)={1\over G}  \boldsymbol{\mathcal{K}}_2(U,D)= 
{1\over G} \cdot { 2\alpha\beta  \over \lambda+\beta-\mu-\alpha+\sqrt{s_1}}. $$
\noindent 
Next we have to compute the stationary horizontal drift of the twisted free process:
$$\tilde{d}^{(1)}=\varphi(U)[\boldsymbol{\mathcal{K}}((x,U),(x+1,U))- \boldsymbol{\mathcal{K}}((x,U),(x-1,U))  ]
+\varphi(D)\boldsymbol{\mathcal{K}}((x,D),(x+1,D))$$
$$={1\over C} {\lambda+\beta+\mu+\alpha-\sqrt{s_1}\over 2}\left(\varphi(U)+\varphi(D)\right) -\varphi(U) {1\over C} {2\lambda\mu \over \lambda+\beta+\mu+\alpha-\sqrt{s_1} }  $$ 
$$= {1\over C} {\lambda+\beta+\mu+\alpha-\sqrt{s_1}\over 2}\cdot 1-{1\over G}\cdot {1\over C}\cdot {\lambda+\beta-\mu-\alpha+\sqrt{s_1}\over 2}  {2\lambda\mu \over \lambda+\beta+\mu+\alpha-\sqrt{s_1} } $$
\begin{equation}\label{eq:drift_d}
={1\over C} \left({\lambda+\beta+\mu+\alpha-\sqrt{s_1}\over 2} -{1\over G}\cdot \lambda\mu\cdot   {\lambda+\beta-\mu-\alpha+\sqrt{s_1} \over \lambda+\beta+\mu+\alpha-\sqrt{s_1} }\right) .
\end{equation}
 

\noindent
The assertion of Proposition \ref{prop:main} follows from the Proposition \ref{prop:mcdonald}, because condition (\ref{condition:pih}) is obviously fulfilled, since the boundary $\triangle$ consists only of two states.

\subsection{Proof of Proposition \ref{prop:main_geometric}}\label{sec:matrix_geometric}

We use the {\bf matrix geometric approach} following Neuts, \cite{Neu_book}.
For a discrete time QBD process as one given in (\ref{eq:matrixP}), Theorem 1.2.1 of Neuts implies that
$$\pi(k,Up)=w_2 \left(e^{\theta_2}\right)^k+w_3\left(e^{\theta_3}\right)^k,$$
where $e^{\theta_2}\geq e^{\theta_3}$  are the eigenvalues of matrix $\mathbf{R}$ described below.
Note   that $\theta_2, \theta_3$ and $w_2, w_3$ depend on $\alpha$. For any $w_2>0$ we have that for $k$ big enough the term $(e^{\theta_2})^k$ dominates $(e^{\theta_3})^k$. However, when $\alpha\to 0$, then $w_2\to 0$  (see Remark on page \pageref{remark:constants}), so that  $w_3(e^{\theta_3})^k$ is the leading term. \par 

\noindent
For matrices $\PP_2, \PP_1, \PP_0$ defined in (\ref{eq:neuts_matrices}) we want to find a matrix $\mathbf{R}=\left[\begin{array}{cc} r_{11} & r_{12}\\ r_{21} & r_{22}\end{array}\right]$ fulfilling
$$\mathbf{R}=\mathbf{R}^2\PP_2+\mathbf{R}\PP_1+\PP_0.$$
We have:
$$\mathbf{R}^2\PP_2+\mathbf{R}\PP_1+\PP_0$$
$$=
\left[\begin{array}{cc}
{(r_{11}^2+r_{12}r_{21})\over C}\mu & 0 \\[4pt]
{(r_{21}r_{11}+r_{22}r_{21})\over C}\mu & 0 \\
\end{array}
\right]
+
\left[\begin{array}{cc}
r_{11}\left(1-{\mu+\lambda+\alpha\over C}\right)+{r_{12}\beta\over C}
			& {r_{11}\alpha\over C}+r_{12}\left(1-{\lambda+\beta\over C}\right) \\[4pt]
r_{21}\left(1-{\mu+\lambda+\alpha\over C}\right)+{r_{22}\beta\over C}
			& {r_{21}\alpha\over C}+r_{22}\left(1-{\lambda+\beta\over C}\right) \\
\end{array}
\right]
+
\left[\begin{array}{cc}
{\lambda\over C}& 0 \\[4pt]
0 & {\lambda\over C} \\
\end{array}
\right],
$$
i.e.
$$
\left[\begin{array}{cc}
r_{11} & r_{12} \\[4pt]
r_{21} & r_{22} \\
\end{array}
\right]
= 
\left[\begin{array}{cc}
{(r_{11}^2+r_{12}r_{21})\over C}\mu+r_{11}\left(1-{\mu+\lambda+\alpha\over C}\right)+{r_{12}\beta\over C}+ {\lambda\over C}
			& {r_{11}\alpha\over C}+r_{12}\left(1-{\lambda+\beta\over C}\right) \\[4pt]
{(r_{21}r_{11}+r_{22}r_{21})\over C}\mu +r_{21}\left(1-{\mu+\lambda+\alpha\over C}\right)+{r_{22}\beta\over C}
			& {r_{21}\alpha\over C}+r_{22}\left(1-{\lambda+\beta\over C}\right) + {\lambda\over C}\\
\end{array}
\right].
$$
One can  check that the solution is
$$\mathbf{R}=\left[
\begin{array}{cc} 
	\displaystyle {\lambda\over\mu} & \displaystyle {\alpha\lambda\over \mu(\lambda+\beta)} \\ [10pt]
       \displaystyle 	{\lambda\over\mu} & \displaystyle {(\alpha+\mu)\lambda\over(\lambda+\beta)\mu} \\
\end{array}
\right]={\lambda\over\mu}\left[
\begin{array}{cc} 
	\displaystyle 1 & \displaystyle {\alpha \over   \lambda+\beta } \\ [10pt]
       \displaystyle 	1 & \displaystyle { \alpha+\mu  \over\lambda+\beta} \\
\end{array}\right].
$$
Eigenvalues of $\mathbf{R} $ are
$$e^{\theta_2}:={\lambda+\mu+\alpha+\beta+\sqrt{s_1}\over 2(\lambda+\beta)}\cdot{\lambda\over\mu},$$
$$e^{\theta_3}:={\lambda+\mu+\alpha+\beta-\sqrt{s_1}\over 2(\lambda+\beta)}\cdot{\lambda\over\mu}.$$

It is easy to check that $e^{\theta_2}=\gamma_1$ (what we already have had) and $e^{\theta_3}=\gamma$. Now,  as $\alpha\to 0$ we have $w_2\to 0$ and thus $\lim_{\alpha\to 0}w_3>0$ (because both limits cannot be equal to 0). The leading term is $w_3(e^{\theta_3})^k$, thus the asymptotic of $\pi(k,Up)$ is $C(Up)\gamma^k$, where  $C(Up)=w_3$. This finishes the proof. \par 
\medskip \par 
\noindent{\bf Remark.} Note, that this method does not give us constant $w_3$ (nor $w_2$, but we already have it, it is $C(Up)$).

\medskip\par\noindent
{\bf Remark.} While looking for parameter $\theta$ in Section \ref{sec:feynman-kac} for which the largest eigenvalue of the Feynman-Kac kernel is equal to 1, we encountered equation (\ref{eq:Wt}). This equation has two solutions: $t_1$ and $t_2$ given in (\ref{eq:t12}). It turns out, that $t_1$ is not the solution for Feynman-Kac kernel, because the right hand side of (\ref{eq:sq_root}) (and (\ref{eq:Wt}) is simply obtained from (\ref{eq:sq_root}) by squaring both sides) is negative. However, $t_2$ is exactly the second term in spectral expansion of $\pi$, what we derived in Section \ref{sec:matrix_geometric} using matrix geometric approach. We conjecture that this can always be  the case for QBD processes.

\subsection{Proof of Proposition \ref{prop:main2p}}
The  asymptotic without constants  is obtained via Proposition \ref{prop:mcdonald} by calculating the harmonic function of the free process and by  verifying  that condition (\ref{condition:pih}), what is done in Section \ref{sec:condition:pih}. 
\subsubsection{The free process.}\noindent
For Model 2 as the boundary we can take $\triangle=\{(0,y,\sigma), y\in \mathbb{N}, \sigma\in\{Up,Down\}\}$. Then the process outside $\triangle$ is shift invariant relative to first coordinate. Define free process
$\mathbf{K}^\infty((m,y,\sigma),(z+m,y',\sigma'))=\mathbf{K}^\infty((0,y,\sigma),(z,y',\sigma'))$, where
 $$
\mathbf{K}^\infty((0,y,\sigma),(z,y',\sigma'))
=\left\{
\begin{array}{lll}
  {\lambda\over C} & \mathrm{for} \ z=0,\ y'=y+1, \ \sigma'=\sigma\\[4pt]
{ \mu\over C}p & \mathrm{for} \ z=-1,\ y'=y \ \mathrm{and} \ \sigma'=\sigma=U & \\[4pt]
{ \mu\over C}(1-p) & \mathrm{for} \ z=-1,\ y'=y+1 \ \mathrm{and} \ \sigma'=\sigma=U & \\[4pt]
 {\mu\over C} & \mathrm{or} \ z=1,\ y'=y-1 \geq 0 \ \mathrm{and} \ \sigma'=\sigma& \\[4pt]

  {\alpha \over C}&\mathrm{for} \ z=0,\ y'=y,\ \sigma=U \ \mathrm{and} \ \sigma'=D & \\[4pt]
  {\beta \over C}&\mathrm{for} \ z=0,\ y'=y,\ \sigma=D \ \mathrm{and} \ \sigma'=U & \\[4pt]
 
 1-{\lambda+\beta\over C} & \mathrm{for} \ z=0, \ y'=y=0 \ \mathrm{and} \ \sigma'=\sigma=D \\[4pt]
 1-{\lambda+\mu +\alpha\over C} &\mathrm{for} \ z=0, \ y'=y=0\ \mathrm{and} \ \sigma'=\sigma=U \\[4pt]
 1-{\lambda+\mu +\beta \over C} &\mathrm{for} \ z=0, \ y'=y\geq 1\ \mathrm{and} \ \sigma'=\sigma=D,  \\[4pt]

 1-{\lambda+2\mu +\alpha \over C} &\mathrm{for} \ z=0, \ y'=y\geq 1 \ \mathrm{and} \ \sigma'=\sigma=U. \\[4pt]
\end{array}
\right. 
$$
\noindent
After removing  the boundary, the free process walks over all $\mathbb{Z}\times\mathbb{N}\times \{Up, Down\}$.

\subsubsection{The harmonic function of the free process}\label{sec:Model2_harmonic_fun}
\begin{lem}\label{lem:free_proc_harmonic2}
The harmonic function of the free process is following: 
$$h(x,y,U)=\left({1\over \gamma_p}\right)^{x+y},\quad h(x,y,D)= \left({1\over \gamma_p}\right)^{x+y} {2\beta\over\lambda+\beta-\mu p-\alpha+\sqrt{s_p} }.$$
\end{lem}
\begin{proof}
For the free process we want to find the harmonic function  of form $h(x,y,\sigma)=e^{\theta_1 x} e^{\theta_2 y} e^{\theta_\sigma}$.
\par\noindent
For $h$ to be the harmonic function for free process we must have
\begin{equation*}
\forall\ (y\in N, \sigma\in\{U,D\} ) \qquad  \sum_{x',y' ,\sigma'} \mathbf{K}^\infty ((0,y,\sigma),(x',y',\sigma'))h(x',y',\sigma')=h(0,y,\sigma).
\end{equation*}
For  $y=0,\ \sigma=U$ we have
$$ \mathbf{K}^\infty ((0,0,U),(0,1,U))h(0,1,U)
+\mathbf{K}^\infty ((0,0,U),(-1,0,U))h(-1,0,U)$$
$$
+\mathbf{K}^\infty ((0,0,U),(-1,1,U))h(-1,1,U)+\mathbf{K}^\infty ((0,0,U),(0,0,D))h(0,0,D)$$ 
$$
+\mathbf{K}^\infty ((0,0,U),(0,0,U))h(0,0,U)=h(0,0,U),$$
$${\lambda\over C}e^{\theta_2}e^{\theta_U}+{\mu\over C}pe^{-\theta_1}e^{\theta_U}+{\mu\over C}(1-p)e^{-\theta_1}e^{\theta_2}e^{\theta_U}
+{\alpha\over C}e^{\theta_D}+\left(1-{\lambda+\mu+\alpha\over C}\right)e^{\theta_U}=e^{\theta_U},$$
i.e.
$$e^{\theta_U}[\lambda+\mu+\alpha-\lambda e^{\theta_2}-\mu p e^{-\theta_1}-\mu (1-p) e^{-\theta_1}e^{\theta_2}]=\alpha e^{\theta_D}$$
Similarly, considering cases $y\geq 1, \ \sigma=U$; $y=0, \sigma=D$ and $y\geq 1, \sigma=D$ we obtain following four equations:
$$\left\{
\begin{array}{lllll}
e^{\theta_U}[\lambda+\mu+\alpha-\lambda e^{\theta_2}-\mu p e^{-\theta_1}-\mu(1-p) e^{-\theta_1}e^{-\theta_2}]&=& \alpha e^{\theta_D},&\qquad &  \\[5pt]
e^{\theta_U}[\lambda+2\mu+\alpha-\lambda e^{\theta_2}-\mu p e^{-\theta_1}-\mu(1-p) e^{-\theta_1}e^{-\theta_2}-\mu e^{\theta_1}e^{-\theta_2}] &=& \alpha e^{\theta_D} ,&\qquad &  \\[5pt]
e^{\theta_D}[\lambda+\beta-\lambda e^{\theta_2} ] &=& \beta e^{\theta_U} ,&\qquad &  \\[5pt]
 e^{\theta_D}[\lambda+\mu+\beta-\lambda e^{\theta_2}-\mu e^{\theta_1}e^{-\theta_2} ]&=&\beta e^{\theta_U}  .&\qquad &  \\[5pt]
\end{array}\right.
$$
First two imply that $e^{\theta_1}=e^{\theta_2}$ and then last two are equivalent. We are left with 2 equations and 3 variables, thus we can set $e^{\theta_U}=1$. Denoting $t=e^{\theta_1}(=e^{\theta_2})$ we have
$$\left\{
\begin{array}{lllll}
 \lambda+\mu+\alpha-\lambda t-\mu p {1\over t} -\mu(1-p)&=& \alpha e^{\theta_D},&\qquad & (i)\\[5pt]
\lambda+\beta-\lambda t   &=& {\beta \over e^{\theta_D}} .&\qquad & (ii)\\[5pt]
\end{array}\right.
$$
Comparing $e^{\theta_D}$ from both equations we have
$${\lambda+\mu+\alpha-\lambda t-\mu p {1\over t} -\mu(1-p)\over \alpha}={\beta\over \lambda+\beta-\lambda t},$$
$$(\lambda+\mu+\alpha-\lambda t-\mu p {1\over t} -\mu(1-p))( \lambda+\beta-\lambda t)=\alpha\beta.$$
Multiplying both sides by $t$ and noting that $t-1$ is one of the solutions, we can rewrite it as
$$(t-1)(\lambda^2 t^2-\lambda(\mu p+\lambda+\alpha+\beta)t+\mu p(\lambda+\beta))=0.$$
Recall that $s_p=(\mu p-\lambda-\beta-\alpha)^2+4\alpha\mu p$. 
The solutions are
$$t_1={\mu p+\lambda+\alpha+\beta + \sqrt{s_p}\over 2\lambda},\quad t_2={\mu p+\lambda+\alpha+\beta - \sqrt{s_p}\over 2\lambda}.$$
Noting that   $s_p=(\mu p+\lambda+\beta+\alpha)^2-4\mu p(\lambda+\beta)$, it can be easily check $t_1>1 \iff \lambda>{\beta\over\alpha+\beta}\mu p$ and $t_2>1 \iff \lambda<{\beta\over\alpha+\beta}\mu p$, i.e. only $t_2$ (which is equal to ${1\over \gamma_p}$) is a valid solution.

 \noindent
From $(ii)$ we have
$$e^{\theta_D}={\beta\over\lambda+\beta-\lambda t_2}
={2\beta\over\lambda+\beta-\mu p-\alpha+\sqrt{  s_p}}.$$
\noindent

\end{proof}
\subsection{Proof of Proposition \ref{prop:main2}}

Since Model 2 with $p=1$ is the special case of general Model 2, we already have the harmonic function given
in Lemma \ref{lem:free_proc_harmonic2}. We can proceed with the twisted free process.
\subsubsection{The twisted free process.}
 
Define the twisted kernel  in the following way:
$\boldsymbol{\mathcal{K}}((m,y,\sigma),(z+m,y',\sigma'))=\boldsymbol{\mathcal{K}}((0,y,\sigma),(z,y',\sigma'))=\boldsymbol{\mathbf{K}}^\infty((0,y,\sigma),(z,y',\sigma')) {h(z,y',\sigma')\over h(0,y,\sigma)}$
 $$
=\left\{
\begin{array}{lllll}
  {\lambda\over C}{h(0,y+1,\sigma)\over h(0,y ,\sigma)} &=& {1\over C} {\lambda+\beta+\mu+\alpha-\sqrt{s_1}\over 2}& \mathrm{for} \ z=0,\ y'=y+1, \ \sigma'=\sigma,\\[8pt]
 
{ \mu\over C} {h(-1,y,U)\over h(0,y,U)} &= & {1\over C}{2\lambda\mu\over \lambda+\beta+\mu+\alpha-\sqrt{s_1}} & \mathrm{for} \ z=-1,\ y'=y \ \mathrm{and} \ \sigma'=\sigma=U, & \\[8pt]
 { \mu\over C} {h(0,y,\sigma)\over h(-1,y+1,\sigma)}&=&{\mu\over C}& \mathrm{for} \ z=1,\ y'=y-1 \geq 0 \ \mathrm{and} \ \sigma'=\sigma,& \\[8pt]

  {\alpha \over C}{h(0,y,D)\over h(0,y,U)}&=&{1\over C} {2\alpha\beta\over\lambda+\beta-\mu-\alpha+\sqrt{s_1}}&\mathrm{for} \ z=0,\ y'=y,\ \sigma=U \ \mathrm{and} \ \sigma'=D, & \\[8pt]
  {\beta \over C}{h(0,y,U)\over h(0,y,D)}&=&{1\over C}{\lambda+\beta-\mu-\alpha+\sqrt{s_1}\over 2}&\mathrm{for} \ z=0,\ y'=y,\ \sigma=D \ \mathrm{and} \ \sigma'=U, & \\[8pt]
 
 (1-{\lambda+\beta\over C}){h(0,0,D)\over h(0,0,D)}&=& 1-{\lambda+\beta\over C}& \mathrm{for} \ z=0, \ y'=y=0 \ \mathrm{and} \ \sigma'=\sigma=D, \\[8pt]
 (1-{\lambda+\mu +\alpha\over C}){h(0,0,U)\over h(0,0,U)}&=&1-{\lambda+\mu +\alpha\over C} &\mathrm{for} \ z=0, \ y'=y=0\ \mathrm{and} \ \sigma'=\sigma=U, \\[8pt]
 (1-{\lambda+\mu +\beta \over C}){h(0,0,D)\over h(0,0,D)}&=&1-{\lambda+\mu +\beta \over C} &\mathrm{for} \ z=0, \ y'=y\geq 1\ \mathrm{and} \ \sigma'=\sigma=D,  \\[8pt]
 (1-{\lambda+2\mu +\alpha \over C}){h(0,0,U)\over h(0,0,U)}&=& 1-{\lambda+2\mu +\alpha \over C}&\mathrm{for} \ z=0, \ y'=y\geq 1 \ \mathrm{and} \ \sigma'=\sigma=U. \\[8pt]
\end{array}
\right. 
$$

\medskip\par\noindent
The transitions of twisted free process are reweighted   transitions of the free process.

\medskip\par\noindent
We are interested in the stationary distribution of the Markovian part of the twisted free process, call it $\boldsymbol{\mathcal{K}}_2$, which state space is $\mathbb{N}\times\{U,D\}$. 
\par 
\noindent 
Denote: 
$$
\lambda'={1\over C}{\lambda+\beta+\mu+\alpha-\sqrt{s_1}\over 2}, \mu'={\mu\over C}, \alpha'={1\over C}{2\alpha\beta\over \lambda+\beta-\mu-\alpha+\sqrt{s_1}}, \beta'={1\over C}{\lambda+\beta-\mu-\alpha+\sqrt{s_1}\over 2}.
$$

\noindent 
The transition of $\boldsymbol{\mathcal{K}}_2$ are
$$
\boldsymbol{\mathcal{K}}_2((y,\sigma),(y',\sigma'))=\left\{
\begin{array}{lllllll}
\lambda' &\mathrm{for} \ y'=y+1\ \mathrm{and}\ \sigma'=\sigma,\\[5pt]
\mu' &\mathrm{for} \ y'=y-1\geq 0\ \mathrm{and}\ \sigma'=\sigma,\\[5pt]
\alpha' &\mathrm{for} \ y'=y,\ \sigma=U\ \mathrm{and}\ \sigma'=D,\\[5pt]
\beta' &\mathrm{for} \ y'=y,\ \sigma=D\ \mathrm{and}\ \sigma'=U,\\[5pt]
1-(\lambda'+\mu'+\alpha') &\mathrm{for} \ y'=y\geq 1\ \mathrm{and}\ \sigma'=\sigma=U,\\[5pt]
1-(\lambda'+\mu'+\beta') &\mathrm{for} \ y'=y\geq 1\ \mathrm{and}\ \sigma'=\sigma=D,\\[5pt]
1-(\lambda'+\alpha') &\mathrm{for} \ y'=y=0\ \mathrm{and}\ \sigma'=\sigma=U,\\[5pt]
1-(\lambda'+\beta') &\mathrm{for} \ y'=y=0\ \mathrm{and}\ \sigma'=\sigma=D.\\
\end{array}\right.
$$
The stationary distribution of $\boldsymbol{\mathcal{K}}_2$ is given by:
$$\varphi(y,U)=B\cdot \left({\lambda'\over\mu'}\right)^y{\beta'\over\alpha'+\beta'},
\quad
\varphi(y,D)=B\cdot \left({\lambda'\over\mu'}\right)^y{\alpha'\over\alpha'+\beta'},\quad
B=1-{\lambda'\over \mu'}=1-{\lambda+\beta+\mu+\alpha-\sqrt{s}\over 2\mu}.
$$
Marginally $\boldsymbol{\mathcal{K}}_2(y,\cdot)$ is a birth and death process with birth rate $\lambda'$ and death rate $\mu'$, the stationary distribution of it is geometric: probability of having $k$ customers equals to  $B_1  \left({\lambda'\over\beta'}\right)^k$ ($B_1$ is  a normalisation constant). Similarly, $\boldsymbol{\mathcal{K}}_2(\cdot,\sigma)$ is a Markov chain with two states, the stationary distribution of which is: ${\beta'\over \alpha'+\beta'}$ of being in $Up$ and  ${\alpha'\over \alpha'+\beta'}$ of being in $Down$ status. Process $\boldsymbol{\mathcal{K}}_2(y,\sigma)$ is not a product of its marginals, but its stationary distribution is of a product form. This can be checked directly, for example for $y\geq 1$ we have:
$$\varphi(y,U)=\sum_{y',\sigma'}\varphi(y',\sigma')\boldsymbol{\mathcal{K}}_2((y',\sigma'),(y,U))$$
since
$$\varphi(y,U)=\varphi(y-1,U)\boldsymbol{\mathcal{K}}_2((y-1,U),(y,U))
+
\varphi(y+1,U)\boldsymbol{\mathcal{K}}_2((y+1,U),(y,U))$$
$$
+\varphi(y,D)\boldsymbol{\mathcal{K}}_2((y,D),(y,U))
+\varphi(y,U)\boldsymbol{\mathcal{K}}_2((y,U),(y,U)),
$$
$$\varphi(y,U)=\varphi(y-1,U)\lambda'
+
\varphi(y+1,U)\mu'+\varphi(y,D)\beta'
+\varphi(y,U)(1-(\lambda'+\mu'+\alpha')),
$$
$$\varphi(y,U)(\lambda'+\mu'+\alpha')=\varphi(y-1,U)\lambda'
+ \varphi(y+1,U)\mu'+\varphi(y,D)\beta',
$$
$$B\cdot \left({\lambda'\over\mu'}\right)^y{\beta'\over\alpha'+\beta'}(\lambda'+\mu'+\alpha')=B\cdot \left({\lambda'\over\mu'}\right)^{y-1}{\beta'\over\alpha'+\beta'}\lambda'
+ B\cdot \left({\lambda'\over\mu'}\right)^{y+1}{\beta'\over\alpha'+\beta'}\mu'+B\cdot \left({\lambda'\over\mu'}\right)^y{\alpha'\over\alpha'+\beta'}\beta',
$$
$$\beta'(\lambda'+\mu'+\alpha')={\mu'\over\lambda'}\beta'\lambda'
+ {\lambda'\over\mu'}\beta'\mu'+\alpha'\beta',
$$
$$\beta'(\lambda'+\mu'+\alpha')=\beta'(\lambda'+\mu'+\alpha').
$$

\noindent 
Next we have to compute the stationary horizontal drift of the twisted free process:
\par 
\noindent 

$$
\begin{array}{ll}
 \tilde{d}^{(2)}=\\
\varphi(0,U)\left[0-\boldsymbol{\mathcal{K}}((0,y,U),(-1,y,U))\right]+
\sum_{y=1}^\infty \varphi(y,U)\left[\boldsymbol{\mathcal{K}}((0,y,U),(1,y-1,U))-\boldsymbol{\mathcal{K}}((0,y,U),(-1,y,U)\right]
\end{array}
$$
$$+\sum_{y=1}^\infty \varphi(y,D)[\boldsymbol{\mathcal{K}}((0,y,D),(1,y-1,D))]$$
$$ \varphi(0,U)\left[0-{1\over C} {2\lambda\mu\over\lambda+\beta+\mu+\alpha-\sqrt{s_1}}\right] +
\left[{\mu\over C}-{1\over C} {2\lambda\mu\over\lambda+\beta+\mu+\alpha-\sqrt{s_1}}\right]\sum_{y=1}^\infty \varphi(y,U) +{\mu\over C}\sum_{y=1}^\infty \varphi(y,D)$$
$$={\mu\over C}\left(\sum_{y=1}^\infty \varphi(y,U)+\sum_{y=1}^\infty \varphi(y,D)\right)-{1\over C} {2\lambda\mu\over\lambda+\beta+\mu+\alpha-\sqrt{s_1}} \sum_{y=0}^\infty \varphi(y,U).$$
We have $\displaystyle\sum_{y=1}^\infty \varphi(y,U)+\sum_{y=1}^\infty \varphi(y,D)=1-\varphi(0,U)-\varphi(0,D)=1-B\cdot {\beta'\over \alpha'+\beta'}-B\cdot {\alpha'\over \alpha'+\beta'}=1-B={\lambda'\over\mu'}$ and
\noindent 
 $\displaystyle \sum_{y=0}^\infty \varphi(y,U)={\beta'\over \alpha'+\beta'}$.
Using the definitions of $\alpha'$ and $\beta'$ we arrive finally at
\begin{equation}\label{eq:drift_d2}
 \tilde{d}^{(2)}=
{1\over C}\left({\lambda+\mu+\beta+\alpha-\sqrt{s_1}\over 2}-{2\lambda\mu(\lambda+\beta-\mu-\alpha+\sqrt{s_1})^2\over (\lambda+\beta+\mu+\alpha-\sqrt{s_1})(4\alpha\beta+(\lambda+\beta-\mu-\alpha+\sqrt{s_1})^2}\right).
\end{equation}


\noindent
Now we make use of the  Proposition \ref{prop:mcdonald}. We postpone verifying the condition (\ref{condition:pih}) to Section \ref{sec:condition:pih}. In our case $A=\{(y,\sigma), y\in\mathbb{N}, \ \sigma\in\{U,D\}\}$.
\par\noindent
For $\sigma=U$ we have
$$\pi(k,y,Up)\sim {\eta^{(2)}\varphi(y,Up)\over \tilde{d}^{(2)} h(k,y,Up)}={\eta^{(2)}\over\tilde{d}^{(2)}}  B \left(\lambda'\over \mu'\right)^y {\beta'\over \alpha'+\beta'} \gamma_1^k\gamma_1^y
={\eta^{(2)}\over\tilde{d}^{(2)}} B \left({\lambda'\over \mu'}\gamma_1\right)^y {\beta'\over \alpha'+\beta'} \gamma_1^k.$$
Noting that ${\lambda'\over \mu'}\gamma_1={\lambda\over\mu}$ and $G={1\over C}(\alpha'+\beta')$ we have
$$\pi(k,y,Up) \sim {\eta^{(2)}\over \tilde{d}^{(2)}}  {1\over G} {\lambda+\beta-\mu-\alpha+\sqrt{s_1}\over 2} B \left(\lambda\over\mu\right)^y\gamma_1^y=C^{(2)}(Up)\left(\lambda\over\mu\right)^y\gamma_1^y. $$
Similarly for $\sigma=D$ we have
$$\pi(k,y,Down)={\eta^{(2)}\varphi(y,Down)\over \tilde{d}^{(2)} h(k,y,Down)}
={\eta^{(2)}\over\tilde{d}^{(2)}}   {\alpha \over G} B  \left({\lambda \over \mu } \right)^y \gamma_1^k
=C^{(2)}(Down) \left({\lambda \over \mu } \right)^y \gamma_1^k.$$
 
\subsubsection{Verification of the assumption of Proposition \ref{prop:mcdonald}.}\label{sec:condition:pih}
For Propositions \ref{prop:main2} and \ref{prop:main2p} to hold, condition (\ref{condition:pih}) must be verified. We show this for a general $p\in(0,1]$. We consider similar network to Model 2, but we do not allow a customer to join the queue at server 1 when the server is in $Down$ status; in this case customer is rerouted again to the queue at server 1. This is a case of unreliable network with rerouting (``the loss regime'', customer is lost to server in $Down$ status, but it is not lost to the network) introduced by Sauer and Daduna (see Sauer, Daduna \cite{SauDad03} or Sauer \cite{SauThesis}).  Namely, when server is in $Up$ status it operates as classical Jackson network, but when it is in $Down$ status the routing is changed, so that with probability 1 customer stays at server 2. This is so-called RS-RD (Random Selection - Random Destination) principle for rerouting. They showed, that then the stationary distribution (say $\pi^{(S)}$) is   a product form of the stationary distribution of pure Jackson network and of the stationary distribution of being in ${Up}$ or $Down$ status. For the above introduced system, the traffic  equation is:
$$\eta_1=\eta_2,\qquad \eta_2=\lambda+\eta_1(1-p).$$
The solution is $\eta_1=\eta_2={\lambda\over p}$. Finally,
$$\pi^{(S)}(x,y,Up)=C^{(S)} \cdot \left({\lambda\over \mu p}\right)^{x+y} {\beta\over\alpha+\beta}, \quad \pi^{(S)}(x,y,Down)=C^{(S)} \cdot \left({\lambda\over \mu p}\right)^{x+y} {\alpha\over\alpha+\beta},$$
where $C^{(S)}$ is   a normalisation constant. It also can be checked directly, that the above is the correct stationary distribution, by checking   that balance equation holds.
\smallskip\par 
Described network differs from Model 2 only by one movement: for $x>0$ and $y>0$ there is a possible transition from $(x,y,Down)$ to $(x+1,y-1,Down)$ for Model 2, but there is no such transition in the above model. Obviously, the stationary distribution $\pi^{(S)}(0,\cdot,\sigma)$ is stochastically greater then $\pi(0,\cdot,\sigma)$, the stationary distribution of Model 2. This can be seen for example by constructing a coupling such that both networks move in the same way, whenever it is possible (when one of the processes is to make forbidden transition - like leaving the state space - it makes no move then), except one transition: when process of Model 2 goes from $(x,y,Down)$ to $(x+1,y-1,Down)$, then the other makes no move.
\par 
Now, for Model 2 as boundary we have $\triangle=\{(0,y,\sigma): y\in\mathbf{N}, \sigma\in\{Up,Down\}\}$ and the harmonic function (given in Lemma \ref{lem:free_proc_harmonic2}) is $h(x,y,\sigma)=C(\sigma) \left({1\over \gamma_p}\right)^{x+y}$. In the  condition (\ref{condition:pih}) we have:
$$\sum_{(x,A)\in\triangle} \pi(x,A)h(x,A)=\sum_{y=0}^\infty \pi(0,y,U)h(0,y,U) + \sum_{y=0}^\infty \pi(0,y,D)h(0,y,D) $$
$$=: E_\pi[h(0,Y,U)]+ E_\pi[h(0,Y,D)] \leq  E_\pi^{(S)}[h(0,Y,U)]+ E_\pi^{(S)}[h(0,Y,D)]$$
since  $h$ is increasing w.r.t. second coordinate and
$$ \pi(0,\cdot ,\sigma) <_{st} \pi^{(S)}(0,\cdot,\sigma), \quad \sigma\in\{Up,Down\}.$$
And for $\pi^{(S)}$ we have
$$ E_\pi^{(S)}[h(0,Y,U)]+ E_\pi^{(S)}[h(0,Y,D)] = \sum_{y=0}^\infty c_1 \left({\lambda\over \mu p}\right)^y \left({1\over \gamma_p}\right)^y  + \sum_{y=0}^\infty  c_2 \left({\lambda\over \mu p}\right)^y \left({1\over \gamma_p}\right)^y$$
with appropriate constants $c_1$ and $c_2$. Of course it is finite if ${\lambda\over \mu p}<\gamma_p$. It can easily be checked, that it holds for any set of parameters. Thus, the condition (\ref{condition:pih}) holds.

\section*{Acknowledgements}
This work was done during my stay in Ottawa as a Postdoctoral Fellow, supported by NSERC grants of David McDonald and Rafa{\l} Kulik.
I would like to thank David McDonald for the whole assistance during writing this paper and Rafa{\l} Kulik for many useful comments and suggestions.

\end{document}